\documentclass[12pt]{amsart}
\usepackage{mathrsfs}
\usepackage{amsmath}
\usepackage{amsfonts}
\usepackage{amssymb}
\usepackage{amsthm}
\usepackage{amstext,amscd}

\theoremstyle{plain}
\newtheorem{theorem}{Theorem}[section]
\newtheorem*{acknowledgement}{Acknowledgements}

\newtheorem{corollary}[theorem]{Corollary}

\newtheorem{definition}[theorem]{Definition}

\newtheorem{lemma}[theorem]{Lemma}

\newtheorem{proposition}[theorem]{Proposition}

\numberwithin{equation}{section}
\newtheorem{remark}{Remark}[section]

\begin{document}
\title{$SL_2(\mathbb{C})$-Character Variety of a Hyperbolic Link and Regulator}
\author{Weiping Li}
\address{Department of Mathematics\\
         Oklahoma State University\\
         Stillwater, OK 74078
         USA}
\email{wli@math.okstate.edu}
\author{Qingxue Wang}
\address{School of Mathematical Sciences \\
         Fudan  University\\
         Shanghai, 200433\\
         P.R. China}
\email{qxwang@fudan.edu.cn}

\begin{abstract}
In this paper, we study the $SL_2(\mathbb{C})$-character variety of
a hyperbolic link in $S^3$. We analyze a special smooth projective
variety $Y^h$ arising from some 1-dimensional irreducible slices on
the character variety. We prove that a natural symbol obtained from
these 1-dimensional slices is a torsion in $K_2({\mathbb C}(Y^h))$.
By using the regulator map from $K_2$ to the corresponding Deligne
cohomology, we get some variation formulae on some Zariski open
subset of $Y^h$. From this we give some discussions on a possible
parametrized volume conjecture for both hyperbolic links and knots.
\end{abstract}

\thanks{Q. Wang was supported by NSFC grant \#10801034.}
\keywords{Character variety, Algebraic K-theory, Chern-Simons
invariant, Hyperbolic links, Volume Conjecture, Regulator map}
\subjclass[2000]{Primary:57M25, 57M27; Secondary:14H50, 19F15}

\maketitle

\section{Introduction}

This is the sequel to our previous work \cite{LW2} on the
generalized volume conjecture for a hyperbolic knot in $S^3$. In
this paper we shall study a hyperbolic link in $S^3$, and extend
several results from the knot case. The main idea is to apply the
regulator map in K-theory to the $SL_2(\mathbb{C})$-character
varieties of hyperbolic links.

For a link $L$ in $S^3$, Kashaev (\cite{Ka1}) introduced a sequence
of complex numbers $\{K_N|N>1, \text{odd integer}\}$, which were
derived from a matrix version of the quantum dilogarithms. Kashaev's
Volume Conjecture (\cite{Ka2}) predicts that for any hyperbolic link
$L$ in $S^3$, the asymptotic behavior of his invariants $\{K_N\}$
regains the hyperbolic volume of $S^3-L$. This was verified for the
figure eight knot (\cite{Ka2}). The Volume Conjecture provides an
intriguing relationship between the quantum invariants and the
hyperbolic volume, but we still do not fully understand it.

For the knot case, Murakami-Murakami (\cite{MM}) showed the Kashaev
invariants $\{K_N\}$ can be identified with the values of normalized
colored Jones polynomial at the primitive $N$-th roots of unity.
From this, a new formulation of the Volume Conjecture states that
the asymptotic behavior of the colored Jones invariants of any
knot equals the Gromov simplicial volume of its complement in $S^3$.
This new version of the volume conjecture bridges the quantum
invariants of the knot with its classical geometry and topology.
However, this formulation does not fit well for links, since it does
not hold for many split links (see \cite{MMOTY}). Hence it is a very
interesting question to see what is really behind the volume
conjecture for links.

Following Witten's $SU(2)$ topological quantum field theory, Gukov
(\cite{Guk}) proposed a complex version of Chern-Simons theory and
generalized the volume conjecture to a $\mathbb{C}^{*}$-
parametrized version with parameter lying on the zero locus of the
$A$-polynomial of the knot. In \cite{LW2}, we constructed a natural
torsion element in $K_2$ of the function field of the curve defined
by the $A$-polynomial. We then showed that the part from the
$A$-polynomial in Gukov's generalized volume conjecture can be
interpreted by the regulator map on this torsion element. In
particular, this implied the Bohr-Sommerfeld quantization condition
posed by Gukov (\cite[Page~597]{Guk}).

It is natural to ask if there exists a parametrized volume
conjecture for links in $S^3$ as Gukov showed for the knot case. This
is the motivation of this paper. Now we have to deal with two
problems for links with more than one component. First, its
$SL_2(\mathbb{C})$-character variety has dimension $>1$, and it is
not clear how to define an $A$-polynomial for such a link, which
will contain the geometric information like volume and Chern-Simons
as in the knot case. Secondly, it is not clear how to relate the
colored Jones polynomials to its $SL_2(\mathbb{C})$-character
variety. In this paper, we shall focus on the first problem for
hyperbolic links. We introduce $n$ curves on the geometric component
of the character variety. From these curves, we obtain an
$n$-dimensional smooth projective variety $Y^h$, where $n$ is the
number of the components of the link. We construct a natural torsion
element in $K_2$ of the function field of $Y^h$. By applying the
regulator map on this torsion element, we get the variation formulae
(Theorem \ref{the3.2}) on some Zariski open subset of $Y^h$. When
the link has one component, it recovers the results for hyperbolic
knots. This suggests that there may exist a parametrized volume
conjecture for hyperbolic links and the $Y^h$ may provide a
replacement for the zero locus of the $A$-polynomial of a knot. For
the second problem we do not know how to deal with it, and only give
some speculations at the end of Section $4$.

On the other hand, using the dilogarithm, Dupont (\cite{Dup})
constructed explicitly the Cheeger-Chern-Simons class associated to
the second Chern polynomial. Apply it to a closed hyperbolic
$3$-manifold $M$, we get a number in $\mathbb{C} / \mathbb{Z}$. He
(loc.cit.) showed that its imaginary part equals the hyperbolic
volume of $M$ and the real part is the Chern-Simons invariant of
$M$. In general, for an odd dimension hyperbolic manifold of finite
volume, Goncharov (\cite{Gon}) constructed an element in the
Quillen's algebraic K-group of $\mathbb{C}$ and proved that after
applying the Borel regulator, we get the volume of the
manifold. In this paper our approach makes use of the regulator map
for the function field of $Y^h$, and it can be regarded as an
analogue of a family version of theirs for the $SL_2(\mathbb{C})$-character
variety of a hyperbolic link.

Our paper is organized as follows. In section 2, we review the
basics of the $SL_2(\mathbb{C})$-character variety of a hyperbolic
link. We then study the properties of a smooth projective variety
$Y^h$ coming from the $1$ dimensional slices of the character
variety. In section 3, we recall the definitions and basic
properties of $K_2$ of a commutative ring. Then we state and prove
our main results in this section. In Section 4, we give some
discussions related to a possible parametrized volume conjecture for
hyperbolic links.

\section{Character Variety of a hyperbolic link}
\subsection{}
Let $L$ be a hyperbolic link in $S^3$ with $n$ components $K_1,
\dots, K_n$. This means that the complement $S^3-L$ carries a
complete hyperbolic structure of finite volume. Let $N(L)$ be an
open tubular neighborhood of $L$ in $S^3$. Put $M_L=S^3-N(L)$, then
it is a compact $3$-manifold with boundary $\partial M_L$ a disjoint
union of $n$ tori $T_1, \dots, T_n$ and it is called the link
exterior. Note that $\pi_1(S^3-L)$ and $\pi_1(M_L)$ are isomorphic.
In the following, we shall identify them.

Let $R(M_L)=\text{Hom}(\pi_1(M_L),SL_2(\mathbb{C}))$ and
$R(T_i)=\text{Hom}(\pi_1(T_i),SL_2(\mathbb{C}))$ ($i=1,\dots,n$) be
the $SL_2(\mathbb{C})$-representation spaces. We have the natural
action of $SL_2(\mathbb{C})$ on them by conjugation. According to
\cite{CS1}, they are affine algebraic sets and so are the
corresponding character varieties $X(M_L)$ and $X(T_i)$, which are
the algebro-geometric quotients of $R(M_L)$ and $R(T_i)$ by
$SL_2(\mathbb{C})$. We then have the canonical surjective morphisms
$t:R(M_L)\longrightarrow X(M_L)$ and $t_i:R(T_i)\longrightarrow
X(T_i)$ which map a representation to its character. Induced by the
inclusions of $\pi_1(T_i)$ into $\pi_1(M_L)$, we have the
restriction map:
\[
  r: X(M_L)\rightarrow X(T_1)\times \cdots \times X(T_n).
\]

For details on the character varieties, we refer to \cite{CS1,
CGLS, CCGLS, Sha}.

\subsection{}
Let $\rho_0: \pi_1(M_L)\rightarrow SL_2(\mathbb{C})$ be a
representation associated to the complete hyperbolic structure on
$S^3-L$. Then it is irreducible. Denote $\chi_0$ its character. Fix
an irreducible component $R_0$ of $R(M_L)$ containing $\rho_0$. Let
$X_0=t(R_0)$, then $X_0$ is an affine variety of dimension $n$
(\cite{CS1,Sha}). We call it a geometric component of the character
variety. We define $Y_0:=\overline{r(X_0)},$ where the bar means the
Zariski closure of the image $r(X_0)$ in $X(T_1)\times \cdots \times
X(T_n)$.

For $g\in \pi_1(M_L)$, there is a regular function
$I_g:X_0\rightarrow \mathbb{C}$ defined by $I_g(\chi)=\chi(g)$, for
all $\chi\in X_0$. The following proposition was proved in
\cite{CS2}.

\begin{proposition}\label{prop1}
Let $\gamma_i$ be a non-contractible simple closed curve in the
boundary torus $T_i$, $1\leq i\leq n$. Let $g_i\in \pi_1(M_L)$ be an
element whose conjugacy class corresponds to the free homotopy class
of $\gamma_i$. Let $k$ be an integer with $0\leq k\leq n$, and let
$V$ be the algebraic subset of $X_0$ defined by the equations
$I_{g_i}^2(\chi)=4$, $k<i\leq n$. Let $V_0$ denote an irreducible
component of $V$ containing $\chi_{\rho_0}$. Then if $\chi$ is a
point of $V_0$, $i$ is an integer with $k<i\leq n$, and $g$ is an
element of the subgroup (defined up to conjugacy)
$\text{Im}(\pi_1(T_i)\rightarrow \pi_1(M_L))$, we have
$I_{g}(\chi)=\pm 2$. Furthermore, if $k=0$, then
$V_0=\{\chi_{\rho_0}\}$.
\end{proposition}

\begin{proof}
See \cite[Proposition~2, Page~539]{CS2}.
\end{proof}

The following is a generalization of the knot case (c.f.
\cite{CS1,CS2}).

\begin{proposition}\label{prop2}
$Y_0$ is an $n$-dimensional affine variety.
\end{proposition}

\begin{proof}
It is clear that $Y_0$ is an affine variety. We need to show that
$\dim Y_0=n$. Since $\dim X_0=n$, $\dim Y_0\leq n$. Assume that
$\dim Y_0=m<n$. Then for $y\in r(X_0)$,  every component of the
fibre $r^{-1}(y)$ has dimension $\geq n-m\geq 1$. Take
$y=r(\chi_0)$, then there is an irreducible component $C$ of the
fibre $r^{-1}(y)$ containing $\chi_0$ and $\dim C\geq 1$. For each
boundary torus $T_i$ and a non-trivial $g_i\in
\text{Im}(\pi_1(T_i)\rightarrow \pi_1(M_L))$, consider the regular
function $I_{g_i}:X_0\longrightarrow \mathbb{C}$. For all $\chi\in
C$, $I_{g_i}(\chi)=I_{g_i}(\chi_0)$. Since $\chi_0$ is the character
of the complete hyperbolic structure on $M_L$,
$I^2_{g_i}(\chi)-4=I^2_{g_i}(\chi_0)-4=0$ for all $\chi\in C$,
$g_i\in \text{Im}(\pi_1(T_i)\rightarrow \pi_1(M_L))$, $1\leq i\leq
n$. Now we fix $n$ non-trivial $g_i\in
\text{Im}(\pi_1(T_i)\rightarrow \pi_1(M_L))$, $1\leq i\leq n$.
Consider the algebraic subset $V$ of $X_0$ defined by the equations
$I^2_{g_i}-4=0$, $1\leq i\leq n$. By its construction, $C$ is
contained in an irreducible component, say $V_0$, of $V$ containing
$\chi_0$. Hence $\dim V_0\geq 1$. On the other hand, by Proposition
\ref{prop1}, $V_0=\{\chi_0\}$, a contradiction. Therefore, $\dim
Y_0=n$.
\end{proof}

For every boundary torus $T_i$, fix a meridian-longitude basis
$\{\mu_i,\lambda_i\}$ for $\pi_1(T_i)=H_1(T_i;\mathbb{Z})$. Given
$1\leq i\leq n$, we define $X_0^i$ as the subvariety of $X_0$
defined by the equations $I_{\mu_j}^2-4=0$, $j\ne i$, $1\leq j\leq
n$. Let $V_i$ be an irreducible component of $X_0^i$ containing
$\chi_0$.

\begin{proposition}\label{prop3}
For each $i=1,\dots,n$, $V_i$ has dimension $1$.
\end{proposition}

\begin{proof}
Since $X_0^i$ is defined by $n-1$ equations and $\dim X_0=n$, every
component of $X_0^i$ has dimension at least $1$. Now assume that
$\dim V_i\geq 2$. Let $U$ be the subvariety of $V_i$ defined by the
equation $I_{\mu_i}^2-4=0$ and let $U_0$ be the irreducible
component of $U$ containing $\chi_0$. Then $\dim V_i\geq 2$ implies
that $\dim U_0\geq 1$. But this contradicts the last assertion in
Proposition \ref{prop1}. Hence, $\dim V_i=1$.
\end{proof}

\begin{lemma}\label{le4}
Given a non-trivial $g_i\in \text{Im}(\pi_1(T_i)\rightarrow
\pi_1(M_L))$, $1\leq i\leq n$, then\\
(1). On every $V_j$ with $j\ne i$, we have $I_{g_i}=\pm2$ is a constant.\\
(2). On $V_i$, $I_{g_i}$ is not a constant, hence it is not a
constant on $X_0$ either.
\end{lemma}

\begin{proof}
$(1)$ follows from the definition of $V_j$ and Proposition
\ref{prop1}.

For $(2)$, suppose $I_{g_i}$ were a constant on $V_i$, then
$I_{g_i}=I_{g_i}(\chi_0)=\pm 2$. Consider the algebraic subset $V$ of
$X_0$ defined by the $n$ equations $I_{\mu_j}^2=4$, ($j\ne i$), and
$I_{g_i}^2=4$. Then $V_i$ is contained in some irreducible component
$V_0$ of $V$ containing $\chi_{\rho_0}$. Hence $\dim{V_0}\geq 1$,
contradicting Proposition \ref{prop1}.
\end{proof}

For each $i=1,\dots,n$, let $p_i$ be the projection map from
$X(T_1)\times \cdots \times X(T_n)$ to the $i$-th factor $X(T_i)$.
Denote by $r_i:X_0\longrightarrow X(T_i)$ the composite of $r$ and
$p_i$. Then we have

\begin{proposition}\label{prop5}
For every $i=1,\dots,n$, the Zariski closure $W_i$ of the image
$r_i(V_i)$ in $X(T_i)$ has dimension $1$.
\end{proposition}

\begin{proof}
It is sufficient to consider the case $i=1$. Since $\dim V_1=1$ and
$r_1$ is regular, $\dim W_1\leq 1$. Assume that $\dim W_1=0$. This
means that $r_1(V_1)$ consists of a single point. Therefore, for any
$g_1\in \text{Im}(\pi_1(T_1)\rightarrow \pi_1(M_L))$, $I_{g_1}$ is a
constant on $V_1$. This contradicts Lemma \ref{le4} part $2$.
\end{proof}

\subsection{}
For $1\leq i\leq n$, denote by $R_D(T_i)$ the subvariety of $R(T_i)$
which consists of the diagonal representations. For such a
representation $\rho$, by taking the eigenvalues of $\rho(\mu_i)$
and $\rho(\lambda_i)$, it is clear that $R_D(T_i)$ is isomorphic to
$\mathbb{C}^{*}\times \mathbb{C}^{*}$. We shall denote the
coordinates by $(l_i,m_i)$. Let $t_{i|D}$ be the restriction of
$t_i$ on $R_D(T_i)=\mathbb{C}^{*}\times \mathbb{C}^{*}$.  Set
$D_i=t_{i|D}^{-1}(W_i)$. By the proof of
\cite[Proposition~3.3]{LW1}, $D_i$ is either irreducible or has two
isomorphic irreducible components. Let $y^i\in D_i$ be the point
corresponding to the character of the representation of the
hyperbolic structure on $S^3-L$. Let $Y_i$ be an irreducible
component of $D_i$ containing $y^i$. Then $Y_i$ is an algebraic
curve. Denote by $\overline{Y_i}$ the smooth projective model of
$Y_i$. Denote $\mathbb{C}(\overline{Y_i})$ the function field of
$\overline{Y_i}$ which is isomorphic to the function field
$\mathbb{C}(Y_i)$ of $Y_i$. Note that when $L$ is a hyperbolic knot
($n=1$), $Y_1$ is the locus of the factor of the $A$-polynomial
corresponding to the geometric component.

We define $Y^h=\overline{Y_1}\times \overline{Y_2}\times \cdots \times
\overline{Y_n}$. Note that it is an $n$-dimensional smooth projective variety.
Let $\mathbb{C}(Y^h)$ be the function field of $Y^h$. For each $i$,
we have the injective morphism
$j_i:\mathbb{C}(Y_i)=\mathbb{C}(\overline{Y_i})\rightarrow
\mathbb{C}(Y^h)$ which is induced by the $i$-th projection from
$Y^h$ to $\overline{Y_i}$. In this way we shall take the
$\mathbb{C}(Y_i)$ as subfields of $\mathbb{C}(Y^h)$.  This also
induces the map $j$ on the K-groups:
\[
  j: \bigoplus_{i=1}^{n}K_2(\mathbb{C}(Y_i))\rightarrow
  K_2(\mathbb{C}(Y^h)).
\]
For $f_i, g_i\in \mathbb{C}(Y_i)$, $i=1,\dots,n$,
$j(\sum_{i=1}^{n}\{f_i,g_i\})= \prod_{i=1}^{n}\{f_i,g_i\}$, where we
identify $f_i, g_i$ as rational functions on $Y^h$ via the injection
$j_i$. Note that in this paper we shall use the multiplication in $K_2$
instead of addition.

\begin{proposition}\label{prop6}
There exists a finite field extension $F$ of $\mathbb{C}(Y^h)$ with
the property that for every $i=1,\dots,n$, there is a representation
\[
  P_i: \pi_1(M_L)\longrightarrow SL_2(F)
\]
such that for $1\leq j\leq n$, if $j\ne i$, the traces of
$P_i(\lambda_j)$ and $P_i(\mu_j)$ are either $2$ or $-2$. If $j=i$,
then
\begin{equation*}
  P_i(\lambda_i)=
  \left[
  \begin{matrix}
    l_i & 0\\
    0 & l^{-1}_i
  \end{matrix}
  \right] \, \text{and} \, \,
P_i(\mu_i)=
  \left[
  \begin{matrix}
    m_i & 0\\
    0 & m^{-1}_i
  \end{matrix}
  \right].
\end{equation*}

\end{proposition}

\begin{proof}
By definition, for each $i$, $W_i$ is the Zariski closure of
$r_i(V_i)$ in $X(T_i)$ and $Y_i$ is mapped dominatingly to $W_i$.
The canonical morphism $t: R_0\rightarrow X_0$ is surjective, so we
can choose a curve $E_i\subset R_0$ such that $t(E_i)$ is dense in
$V_i$. Hence $r_i\circ t:E_i\rightarrow W_i$ is dominating. Then the
function fields $\mathbb{C}(E_i)$ and $\mathbb{C}(Y_i)$ are finite
extensions of $\mathbb{C}(W_i)$. By \cite[Page~115]{CS1}, there is a
tautological representation $p_i: \pi_1(M_L)\rightarrow
SL_2(\mathbb{C}(E_i))$, and for any $g\in \pi_1(M_L)$ the trace of
$p_i(g)$ equals $I_g$. Let $F_i$ be the composite field of
$\mathbb{C}(E_i)$ and $\mathbb{C}(Y_i)$. It is finite over both
$\mathbb{C}(E_i)$ and $\mathbb{C}(Y_i)$. We shall view $p_i$ as a
representation in $SL_2(F_i)$. Since $t(E_i)$ is dense in $V_i$, by
Lemma \ref{le4}, if $j\ne i$, traces of $p_i(\lambda_j)$ and
$p_i(\mu_j)$ are $\pm2$; if $j=i$, traces of $p_i(\lambda_i)$ and
$p_i(\mu_i)$ are non-constant functions on $E_i$.  Since
$p_i(\lambda_i)$ and $p_i(\mu_i)$ are commuting and their
eigenvalues $l_i$, $m_i$ are in $F_i$, the representation $p_i$ is
conjugate in $GL_2(F_i)$ to a representation
\[
  P_i: \pi_1(M_L)\longrightarrow SL_2(F_i)
\]
such that if $j\ne i$, the traces of $P_i(\lambda_j)$ and
$P_i(\mu_j)$ are either $2$ or $-2$. If $j=i$, then
\begin{equation*}
P_i(\lambda_i)=
  \left[
  \begin{matrix}
    l_i & 0\\
    0 & l^{-1}_i
  \end{matrix}
  \right] \, \text{and} \, \,
P_i(\mu_i)=
  \left[
  \begin{matrix}
    m_i & 0\\
    0 & m^{-1}_i
  \end{matrix}
  \right].
\end{equation*}

Fix an algebraic closure $\overline{\mathbb{C}(Y^h)}$ of
$\mathbb{C}(Y^h)$. As above, by viewing $\mathbb{C}(Y_i)$ as a
subfield of $\mathbb{C}(Y^h)$, we can identify the finite field
extension $F_i$ as a subfield of $\overline{\mathbb{C}(Y^h)}$. In
$\overline{\mathbb{C}(Y^h)}$, take the composite of $F_i$ and
$\mathbb{C}(Y^h)$ over $\mathbb{C}(Y_i)$, denoted it by $K_i$. Then
$F_i\subset K_i$ and $K_i$ is a finite extension of
$\mathbb{C}(Y^h)$ because the extension $F_i/\mathbb{C}(Y_i)$ is
finite. Now let $F$ be the composite of the fields $K_1,\cdots, K_n$
in $\overline{\mathbb{C}(Y^h)}$. Then $F$ is a finite extension of
$\mathbb{C}(Y^h)$ since each $K_i$ is. Now compose each $P_i$ with
the embedding $SL_2(F_i)\hookrightarrow SL_2(F)$ and the proposition
follows.
\end{proof}

\section{K-theory and Deligne Cohomology}
First we shall recall the basic definitions of $K_2$ of a
commutative ring $A$. The reference is \cite{Mil}. Let $GL(A)$ be
the direct limit of the groups $GL_n(A)$, and let $E(A)$ be the
direct limit of the groups $E_n(A)$ generated by all $n\times n$
elementary matrices.

\begin{definition}
 For $n\geq 3$, the Steinberg group $St(n,A)$ is the group defined
 by generators $x_{ij}^{\lambda}$, $1\leq i\ne j\leq
 n$, $\lambda\in A$, subject to the following three relations:\\

 (i) $x_{ij}^{\lambda}\cdot x_{ij}^{\mu}=x_{ij}^{\lambda+\mu}$;

 (ii) $[x_{ij}^{\lambda},x_{jl}^{\mu}]=x_{il}^{\lambda\mu}$, for $i\ne l$;

 (iii) $[x_{ij}^{\lambda},x_{kl}^{\mu}]=1$, for $j\ne k$, $i\ne l$.
\end{definition}

We have the canonical homomorphism $\phi_n: St(n,A)\rightarrow
GL_n(A)$ by $\phi(x_{ij}^{\lambda})=e_{ij}^{\lambda}$, where
$e_{ij}^{\lambda}\in GL_{n}(A)$ is the elementary matrix with entry
$\lambda$ in the $(i,j)$ place. Taking the direct limit as
$n\rightarrow \infty$, we get
\[
  \phi:St(A)\rightarrow GL(A).
\]
Its image $\phi(St(A))$ is equal to $E(A)$, the commutator subgroup
of $GL(A)$.
\begin{definition}
 $K_2(A)=\text{Ker}\,\phi$.
\end{definition}

It is well-known that $K_2(A)$ is the center of the Steinberg group
$St(A)$ (See \cite[Theorem~5.1]{Mil}) and there is a canonical
isomorphism $\alpha: H_2(E(A);\mathbb{Z})\rightarrow K_2(A)$ (See
\cite[Theorem~5.10]{Mil}).

\subsection{The Symbol} Let $U$,$V$ be two commuting elements of
$E(A)$. Choose $u,v\in St(A)$ such that $U=\phi(u)$ and $V=\phi(v)$.
Then the commutator $[u,v]=uvu^{-1}v^{-1}$ is in the kernel of
$\phi$. Hence $[u,v]\in K_2(A)$. We can check that it is independent of
the choices of $u$ and $v$, and denote it by $U\bigstar V$.

\begin{lemma}\label{le2.3}
 (1). The construction is skew-symmetric: $U\bigstar V=(V\bigstar
 U)^{-1}$.\\
 (2). It is bi-multiplicative: $(U_1\cdot U_2)\bigstar V=(U_1\bigstar V)\cdot (U_2\bigstar
 V)$.\\
 (3). It is invariant under conjugation: if $P\in GL(A)$, then
 $(PUP^{-1})\bigstar (PVP^{-1})=U\bigstar V$.
\end{lemma}

\begin{proof}
 This is \cite[Lemma~8.1]{Mil}. For (3), we remark that since $E(A)$
 is a normal subgroup of $GL(A)$, the left-hand side of the formula
 makes sense. If $P$,$U$,$V$ are in $GL(n,A)$, then choose $p\in St(A)$
 such that
 \[
   \phi(p)=\left[
   \begin{matrix}

  P    & 0         \\
  0    &P^{-1}

  \end{matrix}
  \right]\in E(A).
 \]
 Now we have $\phi(pup^{-1})=PUP^{-1}$ and $\phi(pvp^{-1})=PVP^{-1}$.
 Hence,
 \[
   [pup^{-1},pvp^{-1}]=p[u,v]p^{-1}=[u,v].
 \]

\end{proof}

Given two units $f$, $g$  of $A$, consider the matrices:
\[
 D_f=\left[
  \begin{matrix}

  f    & 0     &0   \\
  0    &f^{-1} &0   \\
  0    & 0     &1

\end{matrix}
\right] ; \;\;
 D_g^{\prime}=\left[
  \begin{matrix}

  g    & 0     &0   \\
  0    &1      &0   \\
  0    & 0     &g^{-1}

\end{matrix}
\right].
\]
They are in $E(A)$ and commute. Define the symbol
$\{f,g\}:=D_f\bigstar D_g^{\prime}$.

\begin{lemma}\label{le2.4}
(1). The symbol $\{f,g\}$ is skew-symmetric:
$\{f,g\}=\{g,f\}^{-1}$.\\
(2). It is bi-multiplicative: $\{f_1f_2,g\}=\{f_1,g\}\{f_2,g\}$.\\
(3). Denote $diag(f_1,\dots,f_n)$ the diagonal matrix with diagonal
entries $f_1,\dots,f_n$. If $f_1\cdots f_n=g_1\cdots g_n=1$, then
 \[
   diag(f_1,\dots,f_n)\bigstar diag(g_1,\dots,g_n)=\{f_1,g_1\}\{f_2,g_2\}\cdots
   \{f_n,g_n\}.
 \]
 where the right-hand side means the product of the symbols
 $\{f_i,g_i\}$, $1\leq i\leq n$.
\end{lemma}

\begin{proof}
 \cite[Lemma 8.2~ Lemma 8.3]{Mil}.
\end{proof}

Let $F$ be a field. Let $SL(F)$ be the direct limit of the groups
$SL_n(F)$. We know that $SL(F)=E(F)$ and any element of $SL_n(F)$ is
also naturally an element of $E(F)$.

\begin{lemma}\label{le2.5}
 Let $u,t\in F$, then\\

$(1). \; \left[
   \begin{matrix}

  1    & t         \\
  0    & 1

  \end{matrix}
  \right]\bigstar
  \left[
   \begin{matrix}

  1    & u         \\
  0    & 1

  \end{matrix}
  \right]=1.
 $\\

 $(2). \;
   \left[
   \begin{matrix}

  -1    & t         \\
  0    & -1

  \end{matrix}
  \right]\bigstar
  \left[
   \begin{matrix}

  1    & u         \\
  0    & 1

  \end{matrix}
  \right] , \,
   \left[
   \begin{matrix}

  1    & t         \\
  0    & 1

  \end{matrix}
  \right]\bigstar
  \left[
   \begin{matrix}

  -1    & u         \\
  0    & -1

  \end{matrix}
   \right] \; \text{and}
   \left[
   \begin{matrix}

  -1    & t         \\
  0    & -1

  \end{matrix}
  \right]\bigstar
  \left[
   \begin{matrix}

  -1    & u         \\
  0    & -1

  \end{matrix}
  \right]
 $  are 2-torsion in $K_2(F)$.\\

 (3). If $U$ and $V$ are two commuting matrices in $SL_2(F)$ and their
 traces are $2$ or $-2$, then $U\bigstar V$ is  $2$-torsion in $K_2(F)$. In
 particular if both have trace $2$, then $U\bigstar V=1$.
\end{lemma}

\begin{proof}
We shall use the following notations. For $s\in F$,
\[
  M(1,s)=\left[
   \begin{matrix}

  1    & s         \\
  0    & 1

  \end{matrix}
  \right]\;\; \text{and} \;\;
  M(-1,s)=\left[
   \begin{matrix}

  -1    & s         \\
  0    & -1

  \end{matrix}
  \right].
\]
In particular, $M(1,0)$ is the $2\times 2$ identity matrix and
$M(-1,0)$ is the $2\times 2$ diagonal matrix with diagonal entries
$-1$.

For $(1)$, $M(1,t)\bigstar M(1,u)=[x_{12}^{t},x_{12}^{u}]=1$ by the
definition of $St(A)$.

For $(2)$, notice that by the definition, $M(1,0)\bigstar A=1$ and
$A\bigstar A=1$ for any $A\in E(F)$. By Lemma \ref{le2.3},
\[
  1=(M(-1,0)\cdot M(-1,0))\bigstar M(1,s)=(M(-1,0)\bigstar
  M(1,s))^2,
\]
so $M(-1,0)\bigstar M(1,s)$ is a $2$-torsion in $K_2(F)$. Since
\[
  M(-1,t)=M(-1,0)\cdot M(1,-t),\; M(-1,u)=M(-1,0)\cdot M(1,-u),
\]
by Lemma \ref{le2.3} and the first part, we have
\[
  M(-1,t)\bigstar M(1,u)=(M(-1,0)\bigstar M(1,u))(M(1,-t)\bigstar
  M(1,u))=M(-1,0)\bigstar M(1,u)
\]
and
\[
  M(-1,t)\bigstar M(-1,u)=(M(-1,0)\bigstar M(1,-u))(M(1,-t)\bigstar
  M(-1,0)),
\]
hence they are $2$-torsion.

For (3), we can find $P\in GL_2(F)$ such that
\[
  PUP^{-1}=\left[
   \begin{matrix}

  \pm1    & t         \\
  0    & \pm1

  \end{matrix}
  \right]\;\; \text{and} \;\;
  PVP^{-1}=\left[
   \begin{matrix}

  \pm1    & u        \\
  0    & \pm1

  \end{matrix}
  \right].
\]
Then it follows from the first two parts and Lemma \ref{le2.3} (3).
\end{proof}

The following proposition slightly generalizes
\cite[Lemma~4.1]{CCGLS}. The proof is the same.
\begin{proposition}\label{prop2.6}
Let $\pi$ be a free abelian group of rank two with $\{e_1,e_2\}$ its
basis. Let $f:\pi\rightarrow E(A)$ be a group homomorphism defined
by $f(e_1)=U$, $f(e_2)=V$. Then there is a generator $t$ of
$H_2(\pi;\mathbb{Z})$ such that $\alpha(f_{*}(t))=U\bigstar V$,
where $\alpha:H_2(E(A);\mathbb{Z})\rightarrow K_2(A)$ is the
canonical isomorphism and $f_{*}:H_2(\pi;\mathbb{Z})\rightarrow
H_2(E(A);\mathbb{Z})$ is the homomorphism induced by $f$.
\end{proposition}

\begin{proof}
Since $\pi$ is abelian, $U$ and $V$ commute. $U\bigstar V$
is well-defined. Let $F$ be the free group on $\{e_1,e_2\}$. The
homomorphism $f$ gives rise to the following commutative diagram of
short exact sequences of groups:
\[
 \begin{CD}
  0@>>>[F,F]@>>>F@>>>\pi@>>>0\\
  @VVV @Vf_2VV @Vf_1VV @VfVV @VVV \\
  0 @>>> K_2(A) @>>> St(A) @>\phi>> E(A) @>>> 0
 \end{CD}
\]
where $f_2([e_1,e_2])=U\bigstar V$. Apply the homology spectral
sequence to the above diagram, we obtain the following diagram:
\[
  \begin{CD}
  H_2(\pi;\mathbb{Z})@>>>H_0(\pi;H_1([F,F];\mathbb{Z})) \\
  @Vf_{*}VV @VgVV\\
  H_2(E(A);\mathbb{Z})@>\alpha>>K_2(A)
  \end{CD}
\]

The top horizontal arrow is an isomorphism. The class of $[e_1,e_2]$
is the generator of $H_0(\pi;H_1([F,F];\mathbb{Z}))$. It is mapped
to $U\bigstar V$ by $g$ which is induced by $f_2$. Let $t$ be the
generator of $H_2(\pi;\mathbb{Z})$ mapped to the class of
$[e_1,e_2]$. Then we have $\alpha(f_{*}(t))=U\bigstar V$ by the
commutative diagram.
\end{proof}

\begin{corollary}\label{cor2.7}
(1). If $U=diag(u,u^{-1})$ and $V=diag(v,v^{-1})$, where $u,v$ are
units of $A$, then there is a generator $t$ of $H_2(\pi;\mathbb{Z})$
such that $\alpha(f_{*}(t))=\{u,v\}^2$. \\
(2). Suppose $A$ is a field. If $U$ and $V$ are two commuting
matrices in $SL_2(A)$ and their traces are $2$ or $-2$, then the
image of any generator of $H_2(\pi;\mathbb{Z})$ is $2$-torsion in
$K_2(A)$.
\end{corollary}

\begin{proof}
For $(1)$, by Lemma \ref{le2.4}, we have $U\bigstar
V=\{u,v\}\{u^{-1},v^{-1}\}=\{u,v\}^2$.

For (2), by Lemma \ref{le2.5} (3), $U\bigstar V$ is $2$-torsion in
$K_2(F)$.
\end{proof}

Now we can prove the following:

\begin{theorem}\label{thm1}
For each $i=1,\dots,n$, there is an integer
$\epsilon(i)=1\;\text{or}\; -1$, such that the symbol
$\prod_{i=1}^{n}\{l_i,m_i\}^{\epsilon(i)}$ is a torsion element in
$K_2(\mathbb{C}(Y^h))$.
\end{theorem}

\begin{proof}
First, by Proposition \ref{prop6}, for each $i=1,\dots,n$, there
exists a finite extension $F$ of $\mathbb{C}(Y^h)$ and a
representation $P_i: \pi_1(M_L)\rightarrow SL_2(F)$ such that for
$1\leq j\leq n$, if $j\ne i$, then traces of $P_i(\lambda_j)$ and
$P_i(\mu_j)$ are either $2$ or $-2$; if $j=i$, then
\begin{equation*}
P_i(\lambda_i)=
  \left[
  \begin{matrix}
    l_i & 0\\
    0 & l^{-1}_i
  \end{matrix}
  \right] \, \text{and} \, \,
P_i(\mu_i)=
  \left[
  \begin{matrix}
    m_i & 0\\
    0 & m^{-1}_i
  \end{matrix}
  \right].
\end{equation*}

The inclusions of $\pi_1(T_i)$ into $\pi_1(M_L)$ induce the
homomorphisms $\pi_1(T_i)\rightarrow E(F)$ by composing with $P_i$.
This gives rise to the following homomorphisms in group
homology:
\begin{equation}\label{e1}
   \begin{CD}
    \bigoplus_{i=1}^{n}H_2(\pi_1(T_i);\mathbb{Z}) @>\alpha>>
    H_2(\pi_1(M_L);\mathbb{Z}) @>\beta>> H_2(E(F);\mathbb{Z})=K_2(F) \\
   \end{CD},
  \\
\end{equation}

where $\alpha=j_{1*}+\cdots+ j_{n*}$, $\beta=P_{1*}+\cdots+P_{n*}$;
$j_{i*}$ are the morphisms on the group homology induced by the
inclusions $j_i: \pi_1(T_i)\hookrightarrow \pi_1(M_L)$, and $P_{i*}$
are those induced by $P_i$.\\

 The orientation of $M_L$ induces an orientation on
each boundary torus $T_i$. Let $[T_i]$ be the orientation class of
$H_2(T_i; \mathbb{Z})=\mathbb{Z}$. By Corollary \ref{cor2.7} (1),
for each $i$, there is a generator $\xi_i$ of $H_2(\pi_1(T_i))$ such
that $P_{i*}(j_{i*}(\xi_i))=\{l_i,m_i\}^2$. Since $T_i$ is a
$K(\pi_1(T_i),1)$ space, $H_2(\pi_1(T_i); \mathbb{Z})=H_2(T_i;
\mathbb{Z})$. If $\xi_i=[T_i]$, define $\epsilon(i)=1$; if
$\xi_i=-[T_i]$, then define $\epsilon(i)=-1$.

Since $L$ is a hyperbolic link, $M_L$ is a $K(\pi_1(M_L),1)$ space.
Hence we have $H_2(\pi_1(M_L);\mathbb{Z})=H_2(M_L;\mathbb{Z})$.
Under this identification,
\[
 \alpha(\epsilon(1)\xi_1,\cdots, \epsilon(n)\xi_n)=\sum_{i=1}^{n}[T_i]=[\partial{M_L}]=0 \;\;
 \text{in} \;\; H_2(M_L;\mathbb{Z}).
\]
Therefore,
\begin{equation}\label{e2}
  \beta(\alpha(\epsilon(1)\xi_1,\cdots,
 \epsilon(n)\xi_n))=1 \;\;
 \text{in} \;\; K_2(F).
\end{equation}

On the other hand, we have
\begin{alignat}{2}
\beta(\alpha(\epsilon(1)\xi_1,\cdots,
 \epsilon(n)\xi_n))&=\beta(\sum_{i=1}^{n}j_{i*}(\epsilon(i)\xi_i))\notag \\
                   &=\sum_{k=1}^{n}P_{k*}(\sum_{i=1}^{n}j_{i*}(\epsilon(i)\xi_i))\notag \\
                   &=\sum_{i=1}^{n}P_{i*}(j_{i*}(\epsilon(i)\xi_i))+\sum_{1\leq i\ne k\leq n} P_{k*}(j_{i*}(\epsilon(i)\xi_i))\notag \\
                   &=\prod_{i=1}^{n}\{l_i,m_i\}^{2 \epsilon(i)}\cdot \prod_{1\leq i\ne k\leq n}P_k(\mu_i)\bigstar P_k(\lambda_i),\notag
\end{alignat}

where the last step follows from Proposition \ref{prop2.6} and
Corollary \ref{cor2.7}. Note also that we use multiplication in
$K_2(F)$.\\

By Corollary ~\ref{cor2.7} (2), $\prod_{1\leq i\ne k\leq
n}P_k(\mu_i)\bigstar P_k(\lambda_i)$ is  $2$-torsion. Compare with
(\ref{e2}), we see that $\prod_{i=1}^{n}\{l_i,m_i\}^{2 \epsilon(i)}$
is  $2$-torsion in $K_2(F)$. By the same argument as in
\cite[Proposition~3.2]{LW2}, $\prod_{i=1}^{n}\{l_i,m_i\}^{
\epsilon(i)}$ is torsion in $K_2(\mathbb{C}(Y^h))$.
\end{proof}

\begin{remark}
This theorem is a natural generalization of our previous result
\cite[Proposition~3.2]{LW2} about the hyperbolic knot case.
\end{remark}

\begin{remark}
Note that the proof of Theorem \ref{thm1} uses the condition that the
geometric component contains the character $\chi_0$ of the complete
hyperbolic structure. For a non-geometric component of the character
variety, it is not clear whether we can still have the similar
torsion property on it.
\end{remark}

\subsection{Deligne cohomology}
In this subsection we shall recall the definition of Deligne
cohomology, give the construction of the regulator map and apply it
to our situation.\\

Let $X$ be a nonsingular variety over $\mathbb{C}$. First let us
recall the definition of the (holomorphic) Deligne cohomology groups
of $X$. For more details, see \cite{Be,Br,EV}. We define the complex
$\mathbb{Z}(p)_{\mathscr{D}}$ of sheaves on $X$ as follows:
\begin{equation}\label{eq3.2.1}
\mathbb{Z}(p)_{\mathscr{D}}:
 \begin{CD}
 \mathbb{Z}(p)@>>> \mathcal{O}_{X}@>d>> \Omega^{1}_{X}@>d>>  \cdots @>d>> \Omega^{p-1}_{X},
\end{CD}
\end{equation}
where $\mathbb{Z}(p)$ is the constant sheaf $(2\pi
\sqrt{-1})^{p}\mathbb{Z}$ and sits in degree zero, $\mathcal
{O}_{X}$ is the sheaf of holomorphic functions on $X$, and
$\Omega^{i}_{X}$ is the sheaf of holomorphic $i$-forms on $X$. The
first map in (\ref{eq3.2.1}) is the inclusion and $d$ is the
exterior differential. The Deligne cohomology groups of $X$ are
defined as the hypercohomology of the complex
$\mathbb{Z}(p)_{\mathscr{D}}$:
\[
  H^{q}_{\mathscr{D}}(X;\mathbb{Z}(p)):=\mathbb{H}^{q}(X;\mathbb{Z}(p)_{\mathscr{D}}).
\]

For example, the exponential exact sequence of sheaves on $X$
\[
  0\rightarrow \mathbb{Z}(1)\rightarrow \mathcal {O}_{X}\rightarrow \mathcal
  {O}^{*}_{X}\rightarrow 0
\]
gives rise to a quasi-isomorphism between
$\mathbb{Z}(1)_{\mathscr{D}}$ and $\mathcal{O}^{*}_{X}[-1]$, where $
\mathcal{O}^{*}_{X}$ is the sheaf of non-vanishing holomorphic
functions on $X$. Moreover there is a quasi-isomorphism between
$\mathbb{Z}(2)_{\mathscr{D}}$ and the complex (\cite[~page 46]{EV})
\[
 \begin{CD}
   (\mathcal{O}^{*}_{X}@>d \log >> \Omega^{1}_{X})[-1].
 \end{CD}
\]

Therefore, we have for any integer $q$,
\[
 H^{q}_{\mathscr{D}}(X;\mathbb{Z}(1))=H^{q-1}(X;\mathcal{O}^{*}_{X});\;\; H^{q}_{\mathscr{D}}(X;\mathbb{Z}(2))=\mathbb{H}^{q-1}(X;\mathcal{O}^{*}_{X}\rightarrow
\Omega^{1}_{X}).
\] On the other hand, Deligne (\cite{De}) interprets
$\mathbb{H}^{1}(X;\mathcal{O}^{*}_{X}\rightarrow
\Omega^{1}_{X})=H^{2}_{\mathscr{D}}(X;\mathbb{Z}(2))$ as the group
of holomorphic line bundles with (holomorphic) connections over $X$.
For the details, we refer to \cite[Theorem 2.2.20]{Br}.

Let $\mathbb{C}(X)$ be the function field of $X$. Given two
functions $f,g \in \mathbb{C}(X)$, let $D(f,g)$ be the divisors of
the zeros and poles of $f$ and $g$, and let $|D(f,g)|$ denote its support.
Then we have the morphism:
\[
   (f,g):X-|D(f,g)|\longrightarrow \mathbb{C}^{*}\times
   \mathbb{C}^{*},
\]
given by $(f,g)(x)=(f(x),g(x))$.

Let $\mathcal {H}$ be the Heisenberg line bundle with connection on
$\mathbb{C}^{*}\times \mathbb{C}^{*}$. For its construction, see
\cite{Bl} and \cite[~Section 4]{Ram}. Pull back $\mathcal {H}$ along
$(f,g)$ to obtain a line bundle with connection on $X-|D(f,g)|$,
denoted by $r(f,g)$. Hence $r(f,g)\in
\mathbb{H}^{1}(V;\mathcal{O}^{*}_{V}\rightarrow
\Omega^{1}_{V})=H^{2}_{\mathscr{D}}(V;\mathbb{Z}(2))$, where
$V=X-|D(f,g)|$. Moreover we can represent $r(f,g)$ in terms of
\u{C}ech cocycles for
$\mathbb{H}^{1}(V;\mathcal{O}^{*}_{V}\rightarrow \Omega^{1}_{V})$.
Indeed, choose an open covering $(U_i)_{i\in I}$ of $V$ such that
the logarithm of $f$ is well-defined on every $U_i$, denoted by
$\log_{i}f$. Then $r(f,g)$ is represented by the cocyle $(c_{ij},
\omega_i)$, with
\begin{equation}\label{eq3.2.2}
  c_{ij}=g^{\frac{1}{2\pi \sqrt{-1}}(\log_j{f}-\log_i{f})}, \;\; \text{on}\;\,
  U_i\cap U_j;
\end{equation}
\begin{equation}\label{eq3.2.3}
  \omega_i=\frac{1}{2\pi \sqrt{-1}}\log_i{f}\,\frac{dg}{g}, \;\;
  \text{on}\;\, U_i.
\end{equation}
Its curvature is
\begin{equation}\label{eq3.2.4}
 R=\frac{1}{2\pi \sqrt{-1}}\frac{df}{f}\wedge \frac{dg}{g}.
\end{equation}
\begin{remark}
 There is a cup product $\cup$ on the Deligne cohomology groups
 (see \cite{Be,EV}). For $f,g \in
 H^{0}(X;\mathcal{O}^{*}_{X})=H^{1}_{\mathscr{D}}(X;\mathbb{Z}(1))$
 as above, the cup product $f\cup g$ is exactly the line bundle
 $r(f,g)\in H^{2}_{\mathscr{D}}(X;\mathbb{Z}(2))$.
\end{remark}

Furthermore, we have the following properties of $r(f,g)$:
\begin{proposition}\label{pr3.2.1}
$r(f_{1}f_{2},g)=r(f_1,g)\otimes r(f_2,g)$, $r(f,g)=r(g,f)^{-1}$,
and the Steinberg relation $r(f,1-f)=1$ holds if $f\ne 0$, $f\ne 1$.
\end{proposition}

\begin{proof}
See \cite{Bl,EV} and \cite[Section~4]{Ram}. The proofs there assume
that $X$ is a curve. But they are valid for arbitrary $X$ without
change. Note that in order to prove the Steinberg relation, we need
the ubiquitous dilogarithm function.
\end{proof}

Now by the definition of $K_2$ and Proposition \ref{pr3.2.1}, we see
the following
\begin{corollary}
We have the regulator map:
\[
    r: K_2(\mathbb{C}(X))\longrightarrow \underset{U\subset X:\text{Zariski
    open}}{\varinjlim}H^{2}_{\mathscr{D}}(U;\mathbb{Z}(2)),
\]
which maps the symbol $\{f,g\}$ to the line bundle $r(f,g)$.
\end{corollary}

Notice that when $\dim{X}=1$, the line bundle $r(f,g)$ is always
flat, but if $\dim{X}>1$, $r(f,g)$ is not necessarily flat.
Nevertheless, we have the following

\begin{proposition}\label{pro3.9}
If $x\in K_2(\mathbb{C}(X))$ is torsion, the corresponding
line bundle $r(x)$ is flat.
\end{proposition}

\begin{proof}
Let $U$ be the Zariski open subset over which the line bundle $r(x)$
is defined. Since $x$ is torsion in $K_2(\mathbb{C}(X))$, $r(x)$
is torsion in $\mathbb{H}^{1}(U;\mathcal{O}_{U}^{*}\rightarrow
\Omega^{1}_U)$. Choose a suitable open covering $(U_i)_{i\in I}$ of
$U$ such that $r(x)$ is represented by a \u{C}ech cocyle $(c_{ij},
\omega_i)$ with $c_{ij}\in \mathcal {O}^{*}(U_i\cap U_j)$ and
$\omega_i\in \Omega^{1}(U_i)$. Then there exists an integer $n>0$,
such that the class represented by the cocycle $((c_{ij})^n,
n\omega_i)$ is zero. Hence, there exists $t_i\in
\mathcal{O}_{X}^{*}(U_i)$ (or by a refinement covering of $\{U_i\}$)
, such that
\[
  c_{ij}^n=\frac{t_j}{t_i}, \;\; \omega_i=\frac{1}{n}\frac{dt_i}{t_i}.
\]

Therefore, $d\omega_i=0$ for all $i$ and the curvature is $0$.
\end{proof}

Let $|D|$ be the support of the divisors of zeros and poles of the
rational functions $m_i$, $l_i$ on $Y^h$, $1\leq i\leq n$. Define
$Y^h_0=Y^h-|D|$. The line bundle
$r(\prod_{i=1}^{n}\{l_i,m_i\}^{\epsilon(i)})$ is well-defined over
$Y^h_0$.

\begin{corollary}\label{cor3.2.1}
 The line bundle $r(\prod_{i=1}^{n}\{l_i,m_i\}^{\epsilon(i)})$ over
 $Y^h_0$ is flat, therefore it is an element of $H^1(Y^h_0;\mathbb{C}^{*})$.
\end{corollary}

\begin{proof}
 This follows from Theorem~\ref{thm1} and Proposition~\ref{pro3.9}.
\end{proof}

Using the \u{C}ech cocycle for $r(f,g)$ given in (\ref{eq3.2.2}) and
(\ref{eq3.2.3}), we can represent
$r(\prod_{i=1}^{n}\{l_i,m_i\}^{\epsilon(i)})$ as follows. Choose an
open covering $\{U_{\alpha}\}_{\alpha \in \Lambda}$ of $Y^h_0$ such
that on every $U_{\alpha}$, the logarithms of $l_i$ are well-defined
and denoted by $\log_{\alpha}{l_i}$. Then
$r(\prod_{i=1}^{n}\{l_i,m_i\}^{\epsilon(i)})$ is represented by the
cocyle $(c_{\alpha \beta}, \omega_{\alpha})$:
\begin{equation}\label{eq3.2.5}
 c_{\alpha \beta}=\prod_{i=1}^{n}m_i^{{\epsilon(i)}[{\frac{1}{2\pi
 \sqrt{-1}}(\log_{\beta}{l_i}-\log_{\alpha}{l_i})}]},\;\; \text{on} \;
 U_{\alpha}\cap U_{\beta};
\end{equation}

\begin{equation}\label{eq3.2.6}
 \omega_{\alpha}=\sum_{i=1}^{n}\frac{\epsilon(i)}{2\pi \sqrt{-1}}(\log_{\alpha}{l_i})\,\frac{dm_i}{m_i}, \;\;
  \text{on}\;\, U_{\alpha}.
\end{equation}

Let $t_0=(l_1^0,m_1^0,\dots, l_n^0,m_n^0)\in Y^h_0$ be a point
corresponding to the hyperbolic structure of the link complement
$S^3-L$. Then the monodromy of the flat line bundle
$r(\prod_{i=1}^{n}\{m_i,l_i\}^{\epsilon(i)})$ give rises to the
representation $M: \pi_1(Y^h_0,t_0)\rightarrow \mathbb{C}^{*}$. With
its explicit descriptions (\ref{eq3.2.5}), (\ref{eq3.2.6}), we have
the following formula for $M$. Let $\gamma$ be a loop based at
$t_0$. Let $\log{l_i}$ be a branch of logarithm of $l_i$ over
$\gamma-\{t_0\}$, then by a direct calculation we have (c.f.
\cite[~(2.7.2)]{De})

\begin{equation}\label{eq3.2.7}
 M(\gamma)=\exp{\{\sum_{i=1}^{n}(- \frac{\epsilon(i)}{2\pi
 \sqrt{-1}})(\int_{\gamma}\log{l_i}
 \,\frac{dm_i}{m_i}-\log{m_i(t_0)}\int_{\gamma}\frac{dl_i}{l_i})\}}.
\end{equation}

Now we have the main theorem:

\begin{theorem}\label{the3.2}
(i) The real $1$-form $\eta=\sum_{i=1}^{n}\epsilon(i)(\log{|l_i|}\;
d\arg{m_i}-\log{|m_i|}\; d\arg{l_i})$ is exact on $Y^h_0$. Hence
there exists a smooth function $V:Y^h_0\rightarrow \mathbb{R}$ such
that
\[
dV=\sum_{i=1}^{n}\epsilon(i)(\log{|l_i|}\;
d\arg{m_i}-\log{|m_i|}\; d\arg{l_i}).
\]\\
(ii) Suppose $m^0_i=1$, $1\leq i\leq n$. For a loop $\gamma$ with
initial point $t_0$ in $Y^h_0$
\[
   \frac{1}{4
   \pi^{2}}\sum_{i=1}^{n}\epsilon(i)\int_{\gamma}(\log{|m_i|}\, d\log{|l_i|}+\arg{l_i}\, d\arg{m_i})=
    \frac{p}{q},
\]
where $q$ is the order of the symbol
$\prod_{i=1}^{n}\{l_i,m_i\}^{\epsilon(i)}$ in
$K_2(\mathbb{C}(Y^h))$, and $p$ is some integer depending on the
loop $\gamma\in \pi_1(Y^h_0,t_0)$ and the branches of $\arg{l_i}$, $1\leq i
\leq n$.
\end{theorem}

\begin{proof}
First, by (\ref{eq3.2.6}), the curvature of the flat line bundle is
\[
 R=\sum_{i=1}^{n}\frac{\epsilon(i)}{2\pi
 \sqrt{-1}}(\frac{dl_i}{l_i}\wedge \frac{dm_i}{m_i})=0.
\]
On the other hand, we have $d\eta
=\text{Im}(\sum_{i=1}^{n}\epsilon(i)(\frac{dl_i}{l_i}\wedge
\frac{dm_i}{m_i}))$, hence $\eta$ is a real closed $1$-form.

Since the symbol $\prod_{i=1}^{n}\{l_i,m_i\}^{\epsilon(i)}$ has
order $q$ in $K_2(\mathbb{C}(Y^h))$, for a loop $\gamma\in
\pi_1(Y^h_0,t_0)$, by (\ref{eq3.2.7}) we have
\[
  1=M(\gamma)^{q}=(\exp{\{\sum_{i=1}^{n}(- \frac{\epsilon(i)}{2\pi
 \sqrt{-1}})(\int_{\gamma}\log{l_i}
 \,\frac{dm_i}{m_i}-\log{m_i(t_0)}\int_{\gamma}\frac{dl_i}{l_i})\}})^{q}.
\]
Write
$\displaystyle{\sum_{i=1}^{n}\epsilon(i)(\int_{\gamma}\log{l_i}\;\frac{dm_i}{m_i}-\log{m_i(t_0)}\int_{\gamma}\frac{dl_i}{l_i}}=Re+iIm$,
where $Re$ and $Im$ are the real and imaginary parts respectively.
Then we have  $\displaystyle{\exp{(\frac{q\cdot
Im}{2\pi}+\frac{q\cdot Re}{2\pi \sqrt{-1}})}=1}$. Therefore, $Im=0$
and $\displaystyle{\frac{q\cdot Re}{2\pi \sqrt{-1}}=2 \pi \sqrt{-1}
p}$, for some integer $p$. A straightforward calculation or \cite[~Lemma 3.4]{LW2} shows that
\begin{equation}\label{xieta}
 Im=\int_{\gamma}\eta,\;\; Re=-\sum_{i=1}^{n}\epsilon(i)\int_{\gamma}(\log{|m_i|}\,
d\log{|l_i|}+\arg{l_i}\, d\arg{m_i})= \int_{\gamma} \xi.
\end{equation}
These immediately imply both parts of the theorem.
\end{proof}

\begin{remark}
When $n=1$, our $V$ is (up to sign) the volume function of the
representation of the knot complement (\cite{Dun}). For $n\geq 2$,
up to some constant and signs related to the orientations on each
boundary component of the hyperbolic link exterior, the function $V$
should be closely related to the volume function given in
\cite[Theorem 5.5]{Ho}.
\end{remark}

\begin{remark}
From the proof of Theorem \ref{thm1}, the signs $\epsilon(i)$
($1\leq i\leq n$) are determined by the orientation of $M_L$ on its
$n$ boundary tori. For knots, the sign can be neglected since
there is only one term in the $1$-form $\eta$. For links ($n\geq 2$),
if they are not the same, they could have quite
different contributions compared with the knot case. On the other
hand, it is not clear what are the exact geometric meanings of these
signs for the link $L$.
\end{remark}

\begin{remark}
If there exists any representation $\rho: \pi_1(Y^h) \to
GL_n({\mathbb C}), n\geq 2$, then Reznikov (\cite[Theorem 1.1]{Re})
proved that for all $i \geq 2$, the Chern classes  $c_i \in
H^{2i}_{\mathscr D}(Y^h; {\mathbb Z}(i))$ in the Deligne cohomology
groups are torsion.
\end{remark}

\subsection{On the Bohr-Sommerfeld quantization condition for hyperbolic links}
In this subsection, we shall discuss the above Theorem \ref{the3.2}(ii)
from a symplectic point of view. When $n=1$, this is the
Bohr-Sommerfeld quantization condition proposed by Gukov for knots
in \cite[Page~597]{Guk}, and is proved in \cite[Theorem 3.3 (2)]{LW2}.

Let $\Sigma$ be a closed surface with fundamental group $\pi$. Its
$SL_2(\mathbb{C})$-character variety is the space of equivalence
classes of representations from $\pi$ into $SL_2(\mathbb{C})$. This
variety carries a natural complex-symplectic structure, where a
complex-symplectic structure is a nondegenerate closed holomorphic
exterior 2-form (see \cite{Go1, Go2}).

A homomorphism $\rho: \pi \to SL_2(\mathbb{C})$ is irreducible if it
has no proper linear invariant subspace of $\mathbb {C}^2$, and
irreducible representations are stable points, denoted by
$\text{Hom}(\pi, SL_2(\mathbb{C}))^s$. Now $SL_2(\mathbb{C})$ acts
freely and properly on $\text{Hom}(\pi, SL_2(\mathbb{C}))^s$, and
the quotient $X^s(\Sigma) = \text{Hom}(\pi,
SL_2(\mathbb{C}))^s/SL_2(\mathbb{C})$ is an embedding onto an open
subset in the geometric quotient $\text{Hom}(\pi, SL_2(\mathbb{C}))/
/ SL_2(\mathbb{C})$. Thus $X^s(\Sigma)$ is a smooth irreducible
complex quasi-affine variety which is dense in the geometric
quotient (see \cite[Section 1]{Go2}). Note that $\rho$ is a
nonsingular point if and only if $\dim Z(\rho )/Z(SL_2(\mathbb{C}))
= 0$, and this corresponds to the top stratum $X^s(\Sigma)$, where
$Z(u)$ is the centralizer of $u$ in $SL_2(\mathbb{C})$. If $\rho \in
\text{Hom}(\pi, SL_2(\mathbb{C}))$ is a singular point (i.e., $\dim
Z(\rho)/Z(SL_2(\mathbb{C})) >0$), then  all points of $\sigma \in
\text{Hom}(\pi, Z(Z(\rho)))^s$ with stab$(\sigma) = Z(\sigma) =
Z(\rho)$ have the same orbit type and form a stratification of the
$SL_2(\mathbb{C})$-character variety (see \cite[Section 1]{Go1}).

We have the $SL_2(\mathbb{C})$-character variety $X(T^2)$ of the
torus $T^2$ as a surface in ${\mathbb C}^3$ given by
$$x^2+y^2+z^2-xyz-4=0.$$
See \cite[Proposition 3.2]{LW1}.  There exists a natural symplectic
structure on the smooth top stratum $X^s(T^2)$ of $X(T^2)$, and
there exists a symplectic structure $\omega$ on the character
variety $X^s(\partial M_L)=\prod_{i=1}^nX^s(T_i^2)$ such that
$X(M_L)\cap X^s( \partial M_L) (\subset X(M_L))$ is a Lagrangian
subvariety of $X^s(\partial M_L)$, where $X^s(\partial M_L)$ is a
smooth irreducible variety which is open and dense in $X(\partial
M_L)$.

The inclusion $\partial M_L \to M_L$ indeed induces a degree-one map
on the irreducible components. Thus $r(X_0)^s$ (the smooth part of
the image $r(X_0)$) is a Lagrangian submanifold of the symplectic
manifold $X^s(\partial M_L)$. Note that the pullback of the
symplectic 2-form on the double covering of $X^s(T_i^2)$ is again
skew-symmetric and nondegenerate. The symplectic form
$\tilde{\omega}_i$ through the map $t_i$ on the irreducible
component gives the Lagrangian property for the corresponding
pullback of the Lagrangian part $r(X_0^i)^s$. Hence we have the
product Lagrangian smooth part of the pullback of $\prod_{i=1}^n
r(X_0^i)^s$. Then we need to see that the smooth projective model
preserves the Lagrangian and symplectic property.

Let $\tilde{X}(T_i^2)$ be the symplectic blowup of the double
covering of $X(T^2_i)$ as in \cite{MS}. The blowup in the complex
category carries a natural symplectic structure on
$\tilde{X}(T_i^2)$ (\cite[Section 7.1]{MS}). On the other hand, the
corresponding part $\overline{Y}_i$ of $Y_i$ (the irreducible
component of $D_i$ containing $y_i$) lies in the symplectic manifold
$\tilde{X}(T_i^2)$.

Define a compatible Lagrangian blowup with respect to the complex
blowup as following. Define a real submanifold $\tilde{\mathbb R}^n$
of ${\mathbb R}^n\times {\mathbb R}P^{n-1} (\subset {\mathbb C}^n
\times {\mathbb C}P^{n-1})$ as a subspace of pairs $(x, l)$ with $x
= Re(z)\in l$, where $l\in {\mathbb R}P^{n-1}$ is a real line in
${\mathbb R}^n$. If $I_{\mathbb C}$ is complex conjugation on
${\mathbb C}^n$ and $J_{{\mathbb C}P^{n-1}}$ is the complex
involution on ${\mathbb C}P^{n-1}$ given by complex conjugation on
each component, then
$$\tilde{\mathbb R}^n = Fix(I_{\mathbb C} \times J_{{\mathbb C}P^{n-1}}|_{\tilde{\mathbb C}^n})
\subset \tilde{\mathbb C}^n =\{(z_1, \dots, z_n; [w_1: \dots :
w_n]) | w_jz_k=w_kz_j, 1\leq j, k \leq n\}.$$ It is clear that
$\tilde{\mathbb R}^n$ is Lagrangian in $\tilde{\mathbb C}^n$. Hence
the real Lagrangian blowup $\tilde{\overline{Y}}_i$ is Lagrangian in
$\tilde{X}(T_i^2)$, and the Lagrangian submanifold $\tilde{Y}^h$ is
Lagrangian in the symplectic manifold
$\prod_{i=1}^n\tilde{X}(T_i^2)$. This only gives a way to have the
symplectic and Lagrangian properties being preserved under the
blowup, and treat the Lagrangian blowup in a real blowup with
respect to the complex one.

Now we have a Lagrangian submanifold $\tilde{Y_0^h}$ in a symplectic
manifold. Suppose $m^0_i=2$, $1\leq i\leq n$. For a loop $\gamma$
with initial point $t_0$ in $\tilde{Y_0^h}$, by Theorem
\ref{the3.2}(ii)
\[\frac{1}{4 \pi^{2}}\sum_{i=1}^{n}\epsilon(i)\int_{\gamma}(\log{|m_i|}\, d\log{|l_i|}+\arg{l_i}\, d\arg{m_i})=
    \frac{p}{q},\]
where $p$ is some integer and $q$ is the order of the symbol
$\prod_{i=1}^{n}\{l_i,m_i\}^{\epsilon(i)}$ in
$K_2(\mathbb{C}(Y^h))$.  We shall call this result the
Bohr-Sommerfeld quantization condition for hyperbolic links. It
would be interesting to give an interpretation from mathematical
physics, as what Gukov did for hyperbolic knots.

\section{On a possible unified Volume Conjecture for both knots and links}
In this section, we shall give some descriptions and speculations of
a possible parametrized volume conjecture which includes both
hyperbolic knots and links.

By Corollary \ref{cor3.2.1}, the class
$r(\prod_{i=1}^n\{l_i,m_i)^{\varepsilon_i})$ corresponds to a flat
line bundle over $Y^h_0$, therefore the curvature of the holomorphic
connection is zero. Formally this can be expressed as
$d(\xi+\sqrt{-1}\eta)= 0$, where $\xi$ and $\eta$ are defined in
(\ref{xieta}). Hence, $\frac{1}{2 \pi \sqrt{-1}}(\xi+\sqrt{-1}\eta)$
can be viewed as the Chern-Simons \emph{$1$-form} of the line bundle
$r(\prod_{i=1}^n\{l_i,m_i)^{\varepsilon_i})$.

Given a point $p\in Y^h_0$, choose a path $\gamma: [0,1]\rightarrow
Y^h_0$ with $\gamma(1)=p$ and $\gamma(0)=t_0$ a point
corresponding to the complete hyperbolic structure. Write
$\gamma(t)=(l(t), m(t))=(l_1(t), m_1(t), \dots, l_n(t), m_n(t))$.
Recall that $q$ is the order of the symbol $\prod_{i=1}^n\{l_i,
m_i\}^{\varepsilon_i}$ in $K_2(\mathbb{C}(Y^h))$. Let $Vol(L)$ and
$CS(L)$ be the volume and usual Chern-Simons invariant of the
complete hyperbolic structure on $S^3-L$ respectively. Now we define
\begin{equation}\label{vol}
  V(p)=Vol(L)+2\cdot \sum_{i=1}^{n}\epsilon(i) \int_{\gamma} (\log{|l_i|}\;
d\arg{m_i}-\log{|m_i|}\; d\arg{l_i}).
\end{equation}

\begin{equation}\label{scs}
  U(p)=4\pi^2CS(L)+q \cdot \sum_{i=1}^{n}\epsilon(i)\int_{\gamma}(\log{|m_i|}\, d\log{|l_i|}+\arg{l_i}\,
  d\arg{m_i}).
\end{equation}
According to Theorem \ref{the3.2}, we have the quantity
$$
R(p)=\frac{1}{2 \pi}(V(p)+
\frac{\sqrt{-1}}{2\pi}U(p))
$$
is independent of the choices of the path $\gamma$ and it takes
values in $\mathbb{C}/\mathbb{Z}$. We call
$\displaystyle{\frac{1}{4\pi^2}U(p)}$ the \emph{special Chern-Simons
invariant} of the hyperbolic link $L$ at $p$. When $p=t_0$, it
equals $CS(L)$.

\begin{remark}
For $p\ne t_0$, $\displaystyle{\frac{1}{4\pi^2}U(p)}$ is different
from the usual Chern-Simons invariant for a $3$-dimensional
manifold. The latter comes from the transgressive $3$-form of the
second Chern class of the $3$-dimensional manifold.
\end{remark}

In order to formulate a parametrized conjecture parallel to the knot
case as in \cite[Conjecture 3.9]{LW2}, we have to find a way of
relating the quantum invariants to the $n$-dimensional variety
$Y^h_0$ which comes from the $SL_2(\mathbb{C})$ character variety.
By the work of Kashaev and Baseilhac-Benedetti (\cite{BB, Ka1}),
there exists an $SL_2({\mathbb C})$ quantum hyperbolic invariant for
a hyperbolic link in $S^3$, which is conjectured to give the
information of the volume and Chern-Simons at the point for the
complete hyperbolic structure.

Here is a conjectural description. Given a point $p\in Y^h_0$
corresponding to an $SL_2(\mathbb{C})$ representation of
$\pi_1(M_L)$, let's assume that we can define certain quantum
invariants $K_N(L,p)$. Then we formulate the following:\\

\noindent{\bf Conjecture}: (\emph{A Possibly Unified Parametrized
Volume Conjecture})

\begin{equation}\label{uniconj}
\lim_{N\rightarrow \infty} \frac{\log{K_{N}(L,
p)}}{N}=\frac{1}{2 \pi}(V(p)+ \frac{\sqrt{-1}}{2\pi}U(p)).
\end{equation}
\newline

We have the following comments on the conjecture:\\
\begin{remark}
When $L$ is a hyperbolic knot (i.e., $n=1$), $Y^h$ is the smooth
projective model of an irreducible component of the locus of
$A$-polynomial which contains the complete hyperbolic structure. Fix
a number $a$, for $p=(l,m)\in Y^h_0$ with $m=-\exp{(\sqrt{-1}\pi
a)}$ , we take $K_N(L, p)=J_N(L, e^{2\pi \sqrt{-1}\; a/N})$, the
values of the colored Jones polynomial of $L$ evaluated at $e^{2\pi
\sqrt{-1}\; a/N}$. Then our unified Conjecture~\ref{uniconj} is
reduced to \cite[The Reformulated Generalized Volume Conjecture
(3.9)]{LW2} for hyperbolic knots. When $\gamma$ is the constant path
at $t_0$, or equivalently $p=t_0$, it reduces to the
Complexification of Kashaev's Conjecture for hyperbolic knots, see
\cite[~Conjeture 1.2]{MMOTY}.
\end{remark}

\begin{remark}
When $n\geq 2$, one can take $K_N(L, t_0)$ as the the Kashaev and
Baseilhac-Benedetti invariant which is based on the triangulations
of the manifold and is conjectured to give the information of the
volume and Chern-Simons at the complete hyperbolic structure $t_0$
(See \cite[~Section 5]{BB}). For a general $p\in Y^h_0$, although we
expect that there is a way of deforming $K_N(L, t_0)$ to get $K_N(L,
p)$, we do not have a rigorous definition.
\end{remark}

\begin{remark}
If the point corresponding to the hyperbolic structure in $Y_i$ is not smooth, then the point $t_0$ in the definition of (\ref{vol}) and (\ref{scs}) is not unique. If we make different choices of $t_0$, then $V(p)$ and $U(p)$ will be differed by a constant, corresponding to the integrals in (\ref{vol}) and (\ref{scs}) from one choice to another. We can modify the left-hand side of the Conjecture (4.3) by this constant accordingly. So the choice of $t_0$ is not essential and it seems that there is no canonical choice of it.
\end{remark}

\begin{remark}
From the regulator point of view developed in this paper, we expect
there exists a parametrized version of the volume conjecture for
both hyperbolic links and knots.
\end{remark}

\begin{acknowledgement}
Q.Wang is grateful for the support and hospitality of the Marie
Curie Research Programme at DPMMS, University of Cambridge and the
program ANR ``Galois" at the Universit\'{e} Pierre et Marie Curie
(Paris 6) and \'{E}cole Normale Sup\'{e}rieure, Paris. He wants to
thank professors Y. Andr\'{e} and A. Scholl for their helpful
discussions. We thank the referee for his/her helpful comments and suggestions.
\end{acknowledgement}

\end{document}